\documentclass[11pt,runningheads,a4paper]{amsart}
\usepackage[utf8]{inputenc}
\usepackage{amsfonts}
\usepackage{amsthm}
\usepackage{amscd}
\usepackage{t1enc}
\usepackage[mathscr]{eucal}
\usepackage{indentfirst}
\usepackage{graphics}
\usepackage{amssymb}
\usepackage{amsmath}
\usepackage{verbatim}
\usepackage[margin=2.9cm]{geometry}
\usepackage{epstopdf} 
\usepackage[utf8]{inputenc}
\usepackage[english]{babel}
\usepackage{cite}
\usepackage{xcolor}

\usepackage{caption}
\usepackage{subcaption}
\usepackage{algorithm}
\usepackage[noend]{algorithmic}

\usepackage[pdftex]{graphicx}

\newcommand{\zz}{\ensuremath{\mathbb{Z}}}

\theoremstyle{plain}
\newtheorem{Th}{Theorem}[section]
\newtheorem{Lemma}[Th]{Lemma}

\newtheorem{Prop}[Th]{Proposition}
\theoremstyle{definition}

\newtheorem{?}[Th]{Problem}

\newtheorem{Alg}[Th]{Algorithm}
\newenvironment{sproof}{%
  \proof}{\endproof}
\begin{document}
	
	\title{A generalization of Lehman's method}
	\author[J. Hales \and G. Hiary]{Jonathon Hales \and Ghaith Hiary}
	
	\address{
    JH \& GH: Department of Mathematics, The Ohio State University, 231 West 18th
    Ave, Columbus, OH 43210, USA
    }
	\email{hales.41@osu.edu}
	\email{hiary.1@osu.edu}
	\maketitle
	\begin{abstract}
		A new deterministic algorithm for finding square divisors, and finding $r$-power divisors in general, is presented. This algorithm is based on Lehman's method for integer factorization and is straightforward to implement. While the theoretical complexity of the new algorithm is far from best known, the algorithm becomes especially effective if even a loose bound on a square divisor is known. Additionally, we answer a question by D. Harvey and M. Hittmeir on whether their recent deterministic algorithm for integer factorization can be adapted to finding $r$-power divisors.
	\end{abstract}
	
\section{Introduction}\label{Introduction}

Let $N$ be a positive integer. A natural question is whether one can learn partial information about the factorization of $N$ any faster than by factoring $N$ via a general-purpose algorithm. 
We derive a new deterministic algorithm to find square divisors of $N$ or decide that $N$ is squarefree. We also derive a generalization of our algorithm to $r$-power detection. That is, to test for a given $r\ge 2$ if $N$ can be written in the form $N=p^rq$ for some integers $p >1$ and $q$.

We analyze our algorithm rigorously and prove that it runs in $N^{1/(r+2)+o(1)}$ bit operations. The $o(1)$-term accounts for implicit constants that depend on $r$ as well as for the cost of finding integer roots of certain $(r+1)$-degree integer polynomials. Our main theoretical results are stated in \S\ref{main result}.

By examining the big-$O$ implicit constants in the running time of our algorithm, we find that they grow rapidly with $r$, faster than $r^r$. But when $r=2$, these constants are still fairly small. This leads to a practical algorithm for finding square divisors of $N$ in certain special situations. We implement our algorithm and summarize empirical findings about its practical running time in \S\ref{main result}. 

Lastly, we answer a question recently posed by Harvey and Hittmeir in \cite{rRoots}. They asked whether their new
$N^{1/5+o(1)}$ integer factorization algorithm, derived in  \cite{LogLogSpeedUp}, can be adapted to finding $r$-power divisors of $N$. We outline such an adaptation of their algorithm in \S\ref{HH variant}. This adaptation is mainly of theoretical interest.

In the sequel, when we compare (or state) complexity of algorithms, we are comparing (or stating) a worst-case scenario for each algorithm. A worst-case scenario 
in our setting occurs, for instance, when $N$ is prime. Special cases of our algorithms, relying on partial information about the factorization of $N$, will be presented separately.

\section{Preliminaries}\label{Preliminaries}

By calculating $x := \lfloor \sqrt{N}\rfloor$ using integer operations, then checking if $x^2$ is equal to $N$, one can quickly determine if $N$ is a square -- see \cite[Algorithm 9.2.11]{PrimesComp}.
Suppose now that $N$ is not a square, but is divisible by some $p^2>1$. So $N=p^2q$ where $q>1$. Then at least one of $p$ and $q$ is of size $\le N^{1/3}$. Thus,  one can verify that $N$ is squarefree by checking that $N$ is not a square and is not divisible by any integer $n$ satisfying $1<n\le N^{1/3}$.

One can double the speed of this basic trial division algorithm. For if $2$ does not divide $N$, then it suffices to check divisibility of $N$ by odd $n\le N^{1/3}$ only. Extending this idea, by incorporating primes beyond $2$, leads to a variant on trial division known as the wheel factorization method \cite[p. 119]{PrimesComp}. Using the primes $2,3,5$ cuts the number of $n$'s to be checked from about $N^{1/3}$ down to $(4/15) N^{1/3}$. Using more primes produces increasingly marginal improvements and complicates the mechanics of the algorithm, so we do not do this here.

The classical Pollard--Strassen algorithm gives a highly significant improvement on the wheel variant method. Pollard--Strassen, which is based on a multi-point evaluation algorithm for integer polynomials, can test squarefreeness in $N^{1/6+o(1)}$ bit operations and space; see \cite[\S 5.5]{PrimesComp} for an overview.  Using the improved version of Pollard--Strassen derived in \cite{BGS}, the $o(1)$-term stands for a quantity of the form $C (\log N)^2$. 

Recently, Harvey and Hittmier~\cite{rRoots} rigorously analyzed the algorithm in \cite{bdgh} for finding $r$-power divisors, and without any assumption on the factorization of $N$. They assert that their algorithm can find $r$-power divisors in $N^{1/(4r)+o(1)}$ time and negligible space, where the $o(1)$-term stands for a quantity of the form $c_{\epsilon} (\log N)^{10+\epsilon}/r^3$ for any positive constant $\epsilon$. This is currently the best known asymptotic result on this problem. They report, however, that this algorithm is decidedly impractical.

We are also able to adapt the Coppersmith algorithm~\cite{Coppersmith} to search for square divisors of $N$. We sketch this adaptation here since it does not appear elsewhere, as far as we know. The Coppersmith algorithm is a fast algorithm for finding integer roots of bivariate integer polynomials. Using the notation in \cite{coppersmith2}, the bivariate polynomial that arises when searching for square divisors may be taken to be $p(x,y)=(P_0+x)^2(Q_0+y)-N$. Letting $\delta$ be the max degree of $p(x,y)$ in each variable separately, we have $\delta=2$. The Coppersmith algorithm can then quickly find (in poly-log time in the inputs) all integer roots $(x_0,y_0)$ of $p(x,y)$ in the rectangle $|x|< X$ and $|y|<Y$, provided $XY< W^{2/(3\delta)} = W^{1/3}$, where $W$ is the maximum coefficient of $p(xX,yY)$. 

Let us define the bounds $X$ and $Y$ to be $P_0/N^{\alpha}$ and $Q_0/N^{\alpha}$, respectively, where $\alpha>0$ is to be chosen. Considering the coefficient of (say) $x$ in $p(xX,yY)$, we deduce if $P_0$ and $Q_0$ lie in dyadic intervals around $N^{1/3}$, then $W^{1/3}\gtrsim N^{1/3-\alpha/3}$. Thus, the prerequisite condition for the Coppersmith algorithm will be fulfilled if $1/3-\alpha/3 \ge 2/3-2\alpha$, which permits taking $\alpha$ as small as $1/5$. In other words, we can search for $p$ and $q$ in dyadic intervals around $N^{1/3}$ using about $N^{1/5}$ applications of the Coppersmith algorithm, and therefore in $N^{1/5+o(1)}$ time overall. 

Extending this approach to general dyadic intervals containing $P_0$ and $Q_0$, one ultimately obtains an algorithm to find square divisors of $N$, or decide that $N$ is squarefree, in $N^{1/5+o(1)}$ time and negligible space. But the $o(1)$-term is likely very impactful, so we do not pursue this approach here.

To state our findings, we rely on the standard computational model that is nicely outlined in \cite{LogLogSpeedUp}. Moreover, we use the recent important result in \cite{harvey-hoeven} that the cost, M$(n)$, of multiplying two $n$-bit integers satisfies M$(n)=O(n\log n)$. In what follows, $\lg x$ denotes the logarithm of $x$ to base $2$.

\section{Main results}\label{main result}

The computational complexity of a deterministic algorithm  for $r$-power detection can generally be explicitly bounded by a quantity of the form
\begin{equation}
    C_{\alpha} (\log N)^{\kappa_{\alpha}} N^{\alpha},\qquad \textrm{for } N>N_{\alpha}.
\end{equation}
where $C_{\alpha}$, $\kappa_{\alpha}$, and $N_{\alpha}$ may depend on $r$. The smaller $\alpha>0$ is (the more power-saving), the larger $C_{\alpha}$, $\kappa_{\alpha}$, and $N_{\alpha}$ often tend to be, and the harder it is to bound them. 
Consequently, the simplicity of, for example, the wheel variant with primes $2,3,5$ makes it hard to beat in practice unless $N>N_0$, where $N_0$ is likely large. On the other hand, if $N$ is too large, then computing with the given algorithm may no longer be feasible using available computing resources anyway.

Therefore, the target range of $N$ in our context will be $N$ of ``moderate size,'' say with $30$ to $50$ decimal digits. It does not seem that sophisticated general-purpose factoring algorithms are practical over this range of $N$. The main candidate algorithms for squarefree detection over our target range appear to be the wheel variant, the Pollard--Strassen algorithm, the generalized Lehman method derived here, or a hybrid of them. 

We prove the following two theorems, applicable to $r\ge 2$. The first theorem is based on the Lehman strategy for factoring integers \cite{Lehman}, and the second theorem is based on Harvey--Hittmier integer factorization method \cite{LogLogSpeedUp}.
The bound given in each theorem is on the required number of bit operations. 

\begin{Th}\label{LehmanGen}
Algorithm~\ref{LehmansMethodGen} takes as input two integers $N>1$ and $1< r\le \lg N$, and returns a nontrivial factor of $N$ or proves that $N$ is $r$-power free. The algorithm finishes after $O(N^{1/(r+2)}(\log N)^2\log\log N)$ bit operations. 
\end{Th}

\begin{Th}\label{HHGen}
Algorithm~\ref{MainSearchGen} takes as input two integers $N>1$ and $1 < r \le \lg N$, and returns a nontrivial factor of $N$ or proves that $N$ is $r$-power free. The algorithm finishes using $O(N^{1/(3+2r)}(\log N)^5\log\log N)$ bit operations and space.
\end{Th}

Theorem~\ref{LehmanGen} is proved in \S\ref{new algorithm} -- see Proposition~\ref{LehmansMethodProp} -- and a proof of Theorem~\ref{HHGen} is sketched in \S\ref{HH variant}. The space requirement in Theorem~\ref{LehmanGen} is negligible. 

When $r=2$, we analyze Algorithm~\ref{LehmansMethodGen} more thoroughly, detailing our implementation of that case in \S \ref{r=2 case}. We find that our algorithm improves on the simple wheel variant method when $N\ge N_0\approx 4.8 \cdot 10^{43}$, as seen in Table~\ref{timings table}. 

We also implemented Pollard--Strassen using the highly-optimized product tree algorithm in {\tt FLINT}. We found the Pollard--Strassen algorithm to be surprisingly practical, so long as memory was not an obstacle. When $N\approx10^{41}$, for example, the memory requirement in our implementation of Pollard--Strassen would be about $10$ gigabytes. This memory requirement would soon be hard to fulfill, even for otherwise feasible $N$. 

As can be seen from the proof of Proposition~\ref{LehmansMethodProp}, our Algorithm~\ref{LehmansMethodGen} can be made to run significantly faster if we have a good lower bound on $p$, where $N=p^2q$. So, to take advantage of this, we consider the following special problem, which we call Problem $\textrm{P}^*$. 

\begin{quote}
    \textbf{Problem $\textrm{P}^*$}: Given $N>1$. Suppose $N=p^2 q$ where $p$ and $q$ are unknown primes satisfying $q < p < 8q$. Find $p$ and $q$.  
\end{quote}

The constant $8$ in the bound on $p$ and $q$ is somewhat arbitrary. The assumption that $p$ and $q$ are prime is reflective of the experiments we ran and can otherwise be removed.

Problem $\textrm{P}^*$ is closely related to ``factoring with the high bits known'' problem, but in the context of square divisors. We used a small modification of Algorithm~\ref{LehmansMethodGen} to solve Problem $\textrm{P}^*$, as detailed in \S\ref{r=2 case}.
This modification outperforms the Pollard--Strassen algorithm and the wheel variant significantly, which can be seen in Table~\ref{P* table} and Figure~\ref{timings graph}. Since the speed of the Generalized Lehman method depends on how well $p/q$ can be approximated by fractions with small denominators, the timings in Table~\ref{P* table} are the average of $100$ randomly chosen values of $N$ near $10^j$ meeting the criteria for problem $P^*$. 

\begin{table}[ht]
\renewcommand\arraystretch{1.5}
\centering
\begin{tabular}{|c||c|c|c|}
\hline
 $N$ & Generalized Lehman & Pollard--Strassen & Wheel variant \\
 \hline
 $\approx 10^{21}$ & $0.018248$s& $0.1475$s &$0.0993$s \\
 $\approx 10^{24}$ &$0.8226$s & $0.6623$s & $0.8472$s\\
 $\approx 10^{27}$ &$0.3408$s & $3.0142$s & $7.2905$s\\
 $\approx 10^{30}$ &$1.2457$s & $14.09132$s&$78.2866$s \\
 $\approx 10^{33}$ &$5.5649$s & $70.9481$s &$784.1740$s\\
 $\approx 10^{36}$ &$38.2918$ & $248.102$s &$8243.5654$s\\
\hline
\end{tabular}
\caption{\small Timings  for generalized Lehman, Pollard--Strassen, and the wheel variant on Problem $\textrm{P}^*$. These algorithms are implemented in {\tt C++}, {\tt GMP/MPFR}, and {\tt FLINT}. The timings are obtained using $1$ processor with $2.2$ Ghz processor speed.}
\label{P* table}
\end{table}

\begin{table}[ht]
\renewcommand\arraystretch{1.5}
\centering
\begin{tabular}{|c||c|c|c|}
\hline
 $N$ & Generalized Lehman & Pollard--Strassen & Wheel variant \\
 \hline
 $\approx 10^{20}$ & $6.364$s   & $0.285$s & $0.051$s \\
 $\approx 10^{25}$ & $90.233$s  & $2.004$s & $2.017$s    \\
 $\approx 10^{30}$ & $1597.14$s  & $16.986$s & $105.822$s  \\
 $\approx 10^{35}$ & $22774.42$s  & $150.378$s & $4920.82$s    \\
 $\approx 10^{37}$ & $72029.07$s& $434.647$s & $22751.65$s  \\
 $\approx 10^{44}$ & $4.05\cdot 10^{6}\text{s}^*$& $7637.78\text{s}^*$& $4.9\cdot 10^{6}\text{s}^*$\\
\hline
\end{tabular}
\caption{\small Worst-case timings (e.g.\ when $N$ is prime) for generalized Lehman, Pollard--Strassen, and the wheel variant. Entries marked with $*$ are extrapolated.}
\label{timings table}
\end{table}

\begin{figure}[ht]
    \includegraphics[scale = 0.60]{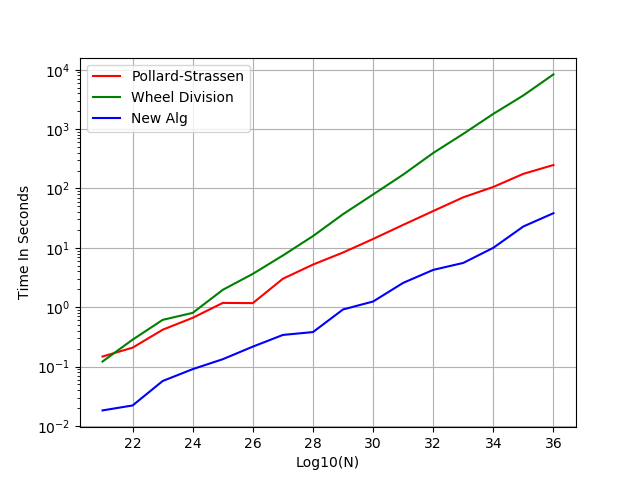}
    \caption{\small Log-log plot of timings for generalized Lehman,  Pollard--Strassen, and the wheel variant on problem $\textrm{P}^*$. Space requirement of Pollard--Strassen algorithm not displayed.}
    \label{timings graph}
\end{figure}

\section{Lehman's Method}\label{Lehman method}

In 1974, Lehman~\cite{Lehman} found a remarkably simple algorithm for factoring a semi-prime $N = pq$. Lehman's idea was to search for integers $a$ and $b$ so that $aq$ is close to $bp$. Given such $a$ and $b$, Fermat's factorization method will quickly recover the factors of $4abN$ closest to the square root, namely $aq$ and $bp$. On the other hand, comparing the arithmetic and geometric means of $aq$ and $bp$ gives that $aq+bp$ is close to $(4abN)^{1/2}$. More specifically, and using the version of Lehman's lemma in \cite[Lemma 3.3]{Harvey}, if $(N/s)^{1/2} \leq p \leq N^{1/2}$, then  there exists $a$ and $b$ such that $1\le ab\le s$ and the difference between the arithmetic and geometric means satisfies
\begin{equation}
0\leq aq+bp-(4abN)^{1/2} \leq \frac{N^{1/2}}{4s(ab)^{1/2}}=:U.
\end{equation}
Thus, for a given admissible $a$ and $b$, and since the geometric mean $(abN)^{1/2}$ is known, one can find the unknown quantity $aq+bp$ by checking $\le U+1$ candidate values close to $(4abN)^{1/2}$. Specifically, for each candidate $u$, it can be quickly determined if $u=aq+bp$ by finding the integer roots of the quadratic polynomial $x^2-ux+abN$, if there are any. For if $u=aq+bp$, then the polynomial $x^2-ux+abN$ has roots $aq$ and $bp$. Once these roots are found, $N$ can be factored quickly.

To ensure that $aq$ is close to $bp$, the Lehman method searches for a good rational approximation $a/b$ of $p/q$. Such an approximation is found by trying representatives from an appropriate subset of the rationals in $[0,1]$, for example those in the Farey sequence. Overall, to minimize the total number of $u$'s  to be checked, the Lehman method chooses $s\approx N^{1/3}$. This yields an algorithm for factoring $N$ in $O(N^{1/3}\log N\log \log N)$ bit operations.

\section{Generalized Lehman algorithm} \label{new algorithm}
Let $N$ and $r$ be integers such that $2\le r\le \lg N$. We would like to generalize Lehman's method to check if $N=p^rq$ for some integers $q$ and $p>1$. If indeed $N=p^rq$ and $N$ is not an $r$-power, then there is a divisor of $N$ less than or equal to $N^{1/(r+1)}$. The key point under this scenario  is that, compared to the usual Lehman method, the intervals containing candidates $u$ of the unknown $aq+rbp$ will be shorter. 

We organize our search for a rational approximation $a/b$ of $p/q$, and hence our search of the candidate $u$'s, by assuming at each step that $p$ is contained in a dyadic interval. This allows us to bound the possible values of $a$ and $b$ more optimally.

\begin{Prop}\label{LehmanInqGen}
	Let $j$ be a nonnegative integer and $B$ be defined as in \eqref{Bdef}. Suppose $B>1$, 
	$N = p^r q$,  and 
	\begin{equation}\label{DyadicInt}
	2^j N^{1/(r+2)} \leq p \leq 2^{j+1}N^{1/(r+2)}.
	\end{equation}
	Then there exists a pair of positive integers $(a,b)$ such that $a \leq \lceil 2^{(r+2)j/3+r+1}\rceil$,  $b\leq N^{1/(r+2)}2^{-(2r+1)j/3}$, and
	\begin{equation}\label{GenPowAMGMBound}
	 aq+rbp-(r+1)(ab^rN)^{1/(r+1)} < \frac{r^{r+3/2}}{\sqrt{ab}}N^{1/(2(r+2))}2^{(1-r)j/6-1}.
	\end{equation}
	
\end{Prop}

\begin{proof}
	We approximate $p/q$ by a fraction $a/b$ in the Farey sequence of order $n$, so $1\leq b \leq n$. We choose $n := \lfloor B\rfloor$ where
\begin{equation}\label{Bdef}
	B := N^{1/(r+2)}2^{-(2r+1)j/3}.
\end{equation}
	This gives (see \cite{HardyWright}), 
	\begin{equation}\label{FareyBound}
	\left|\frac{a}{b}-\frac{p}{q}\right| \leq \frac{1}{(n+1)b} < \frac{1}{bB}.
	\end{equation}
	So both $a$ and $b$ satisfy the bounds stated in the lemma.	We observe that 
	\begin{equation}
	a = \frac{a}{b}b < \left(\frac{p}{q} + \frac{1}{bB}\right)b \leq \frac{p}{q}B + \frac{1}{B} = \frac{p^{r+1}}{N}B+\frac{1}{B} \leq \lceil 2^{(r+2)j/3 + (r+1)}\rceil,
	\end{equation}
	where we used the assumption $B>1$ in the last inequality.
	Next, we bound the difference between the arithmetic and geometric means. From the well-known identity 
	\begin{equation}
	    (w-z)\sum_{k=0}^r w^{r-k}z^k= w^{r+1}-z^{r+1},
	\end{equation}
	applied with $w=aq+rbp$ and $z=(r+1)(ab^r N)^{1/(r+1)}$, we obtain
	\begin{equation}\label{Ptemp1}
	aq+rbp - (r+1)(ab^rN)^{1/(r+1)} = \frac{(aq+rbp)^{r+1} - (r+1)^{r+1}ab^r N}{\sum_{k=0}^{r} (aq+rbp)^{r-k}(r+1)^k\left(ab^r N\right)^{k/(r+1)}}.
	\end{equation}
	Let us define
	\begin{equation}\label{Pdef}
	    P(x,y) = \frac{(x+ry)^{r+1} - (r+1)^{r+1}(xy^r)}{(x-y)^2}.
	\end{equation}
	So $P(aq,bp)$ is equal to the numerator in \eqref{Ptemp1} divided by $(aq-bp)^2$.
	We now observe the identity
	\begin{equation}\label{Pid}
	 P(x,y)= \sum_{k=0}^{r-1} a_k x^{r-1-k}y^k
	\end{equation}
	where the coefficients $a_n$ are defined by the initial conditions $a_{-1} = 0, a_0 = 1$ and the recurrence 
	\begin{equation}
	a_k = 2a_{k-1}-a_{k-2} + \binom{r+1}{k}r^k,\qquad\qquad(k\ge 1).
	\end{equation}
	The identity \eqref{Pid} is routine to verify, e.g.\ by multiplying the right-side by $(x-y)^2$ and expanding. Using positivity to drop all the terms in the denominator in \eqref{Ptemp1} except the $k=0$ term, gives
	\begin{equation}
	aq+rbp - (r+1)(ab^rN)^{1/(r+1)}< \frac{(aq-bp)^2 P(aq,bp)}{(aq+rbp)^r}.
	\end{equation}
	
	Now, by induction, the coefficients $a_k$ are increasing in $k$, and setting $x=0$ in \eqref{Pdef}, we see that $a_{r-1} = r^{r+1}$. Therefore,
	\begin{equation}\label{Pbound}
	\frac{P(aq,bp)}{(aq+rbp)^{r-1}} <\frac{r^{r+2}\max(aq,bp)^{r-1}}{(aq+rbp)^{r-1}} < r^{r+2}.
	\end{equation}
	Additionally, we deduce from inequality $\eqref{FareyBound}$ that
	$|aq-bp| < q/B$. Using these observations together with the AM-GM inequality we obtain from \eqref{Pbound} that
	\begin{equation}\label{Ubound}
	aq+rbp - (r+1)(ab^rN)^{1/(r+1)} <  r^{r+2}\frac{q^2/B^2}{aq+rbp}
	\leq r^{r+2}\frac{q^2/B^2}{2\sqrt{aqrbp}}=\frac{r^{r+3/2}}{\sqrt{ab}}\cdot \frac{q^{3/2}}{2B^2 p^{1/2}}.
	\end{equation}
	Finally, replacing $q$ with $N/p^r$ and using the definition of $B$, we obtain the last quantity is
	\begin{equation}
	\begin{split}
	 &=\frac{r^{r+3/2}}{\sqrt{ab}}\cdot \frac{1}{B^2}\cdot \frac{N^{3/2}}{2p^{(3r+1)/2}}\\
	 &\leq \frac{r^{r+3/2}}{\sqrt{ab}}\cdot \frac{2^{(4r+2)j/3}}{N^{2/(r+2)}}\cdot \frac{N^{3/2}}{2^{(3r+1)j/2+1}N^{(3r+1)/(2(r+2))}}\\
	&\le \frac{r^{r+3/2}}{\sqrt{ab}}\cdot 2^{(1-r)j/6-1}\cdot N^{1/(2(r+2))}.
	\end{split}
	\end{equation}
		\qedhere
\end{proof}

As in the usual Lehman method, we can quickly test if a candidate $u$ is equal to $aq+rbp$, this time by determining if the $(r+1)$-degree polynomial $f_r(x):=x^{r+1}-ux^r + r^rab^rN$ has any integer root with a nontrivial gcd with $N$. This can be determined quickly, as detailed at the end of the proof of Proposition~\ref{LehmansMethodProp} next.

 \begin{Alg}\label{LehmansMethodGen}\mbox{}
 
 \vspace{3mm}
 \noindent
	\textit{Input}: An integer $N>1$ and an integer $2\le r\le \lg N$.
	\vspace{2mm}
	
\noindent	
	\textit{Output}: Either a nontrivial factor of $N$ is returned or that ``$N$ is $r$-power free.''
	\vspace{2mm}
	\begin{algorithmic}[1]
				\STATE Test if $N$ is an $r$-power in by testing if $\lfloor N^{1/r}\rfloor ^r = N$.
		\FOR{$n =2 ,\dots, \lfloor N^{1/(r+2)}\rfloor$}
		\STATE If $n\mid N$ return $n$.
		\ENDFOR 
		\FOR{$j = 0,\dots, \lfloor \frac{1}{r(r+2)}\lg N\rfloor$}
		\FOR{$a = 1,\dots, \lceil2^{((r+2)j/3+r+1}\rceil$}
		\FOR{$b = 1,\dots, \lfloor N^{1/(r+2)}2^{-(2r+1)j/3}\rfloor$}
		\FOR{$u = 1,\dots, \lfloor r^{r+3/2} N^{1/(2(r+2))}2^{-(r-1)j/6-1}(ab)^{-1/2}\rfloor$}
		\STATE Determine if $u = aq+rbp$. If successful, recover $p$ and $q$ and return. 
		\ENDFOR 
		\ENDFOR
		\ENDFOR
		\ENDFOR
		\STATE Return ``$N$ is $r$-power free.''
	\end{algorithmic}
\end{Alg}
\begin{Prop}\label{LehmansMethodProp}
	Algorithm \ref{LehmansMethodGen} is correct, and runs in time $O(N^{1/(r+2)}(\lg N)^2\lg\lg N))$.
\end{Prop}
\begin{proof}
	Let us first show correctness.
	We observe that if line $1$ and the {\tt for} loop in lines $2-3$ finish without finding a nontrivial factor of $N$ (and, thus, without the algorithm reaching an endpoint), then we will know that $N$ is not an $r$-power and has at most $r+1$ prime factors counting multiplicity, and moreover that each factor of $N$ is $> N^{1/(r+2)}$. We may suppose all this about $N$ for the rest of the proof since the completion time of lines $1-3$ is within our target complexity anyway. 
	
	We claim that if $N$ is of the form $p^rq$ with $p,q>N^{1/(r+2)}$, then lines $4-8$ will successfully find $p$ and $q$. Let us first note that all the $j$'s considered in line $4$ will result in value of $B$, as defined in \eqref{Bdef}, satisfying the condition $B>1$ required by Proposition~\ref{LehmanInqGen}.
	
	Now, suppose $N = p^rq$ with $p,q>N^{1/(r+2)}$. Let $j^*\ge 0$ be the integer defined by 
	    	$2^{j^*}N^{1/(r+2)}\leq p < 2^{j^*+1}N^{1/(r+2)}$.
	Since $N^{1/(r+2)} < q=N/p^r$, we deduce the upper bound $p < N^{\frac{r+1}{r(r+2)}}$, and therefore
	\begin{equation}
	    	j^* < \lg N^{\frac{r+1}{r(r+2)}}-\lg N^{\frac{1}{r+2}} = \frac{1}{r(r+2)}\lg N.
	\end{equation}
	Hence, the {\tt for} loop on line $4$ will include $j^*$. When we reach $j=j^*$ in that loop, Proposition~\ref{LehmanInqGen} guarantees there are integers $a^*$, $b^*$, and $u^*$ with $a^* < \lceil2^{((r+2)j^*/3+r+1}\rceil$, $b^* < N^{1/(r+2)}2^{-(2r+1)j^*/3}$, and $u^* < \lfloor r^{r+3/2} N^{1/(2(r+2))}2^{-(r-1)j^*/6-1}(ab)^{-1/2}\rfloor$, and such that $u^* = a^*q + rb^*p$. Thus, $a^*,b^*,u^*$ will be found during the search in lines $5-7$. And once $a^*,b^*,u^*$ are found, line $8$ will  successfully recover $p$ and $q$.  
	
	If no $a^*,b^*,u^*$  are ever found, then we will reach line $9$ which will return that $N$ is $r$-power free. This will be the correct outcome in this case because if we reach line $9$, then $N$ cannot be written in the form $p^rq$ with $p,q>N^{1/(r+2)}$. This, together with the earlier suppositions about $N$ (based on information from lines $1-3$), will imply $N$ is $r$-power free.
	
	We now consider the time complexity. Lines $1-3$ complete in $O(N^{1/(r+2)})$ divisions, so using $O\left(N^{1/(r+2)}\lg N\lg\lg N\right)$ bit operations. The {\tt for} loops in lines $4-8$ complete after  
	\begin{equation}
	    \begin{split}
	&\le \sum_{j=0}^{\lfloor \frac{1}{r(r+2) }\lg N\rfloor}\sum_{a=1}^{\lceil 2^{(r+2)j/3+r+1}\rceil}\sum_{b=1}^{\lfloor B\rfloor}\left \lfloor \frac{r^{r+3/2}N^{1/(2(r+2))}}{\sqrt{ab}}\right\rfloor \\
	&\qquad\qquad\leq  r^{r+3/2}N^{1/(2(r+2))}\left(\sum_{j=0}^{\infty} \sum_{a=1}^{\lceil 2^{(r+2)j/3+r+1}\rceil}\frac{1}{\sqrt{a}}\sum_{b=1}^{\lfloor B\rfloor}\frac{1}{\sqrt{b}}\right)\\
	&\qquad\qquad=r^{r+3/2}N^{1/(2(r+2))}\sum_{j=0}^{\infty}O\left(\sqrt{2^{(r+2)j/3+r+1}}\right)O\left(\sqrt{B}\right),
		    \end{split}
	\end{equation}
	iterations. Thus, using the definition of $B$ in the final sum, this is bounded by
	\begin{equation}
	    \begin{split}
	&r^{r+3/2}N^{1/2(r+2)}O\left(\sum_{j=0}^{\infty}2^{(r+2)j/6+(r+1)/2}N^{1/2(r+2)}2^{-(2r+1)j/6} \right)\\
	&\qquad\qquad= 	r^{r+3/2}N^{1/2(r+2)}O\left( N^{1/2(r+2)}\sum_{j=0}^{\infty}2^{(1-r)j/6}\right)\\
	&\qquad\qquad=O(N^{1/(r+2)}),
		    \end{split}
	\end{equation}
	where the final inequality follows because the sum over $j$ is convergent for any $r\ge 2$. 
	
	Finally, as remarked in the paragraph preceding Algorithm~\ref{LehmansMethodGen}, we can execute line $8$ by finding the integer roots of $f_r(x)=x^{r+1}-x^ru+r^rab^rN$, if there are any.  For if $N = p^rq$ and $u = aq+rbp$, then $rbp$ is a root of $f_r(x)$.
	
	To this end, we first note that $f_r(x)$ has at most three real roots. This is  seen on observing that $f'_r(x) = x^{r-1}((r+1)x-ru)$ has exactly two real roots, namely $0$ and $\rho=ru/(r+1)$. And since $f_r(0) > 0 > f_r(\rho)$, two of these roots are positive.\footnote{If $r=2$ (or is even), then $f_r(x)$ has exactly one negative root and two positive roots.}
	
	Now, if $u=a q+rb p$, then $rb p$ is a root of $f_r(x)$, and  we have 
\begin{equation}
\left|\frac{r u}{r+1}-rb p\right| = \frac{r}{r+1}|a q-b p| = \frac{rqb}{r+1}\left|\frac{a}{b}-\frac{p}{q} \right| < \frac{r}{r+1} \frac{q}{B} < \frac{r}{r+1}N^{1/4}2^{-j/3}=: L.
\end{equation}
Thus, using a bisection root-finding method on the intervals $[ru/(r+1)-L,ru/(r+1)]$ and $[ru/(r+1),ru/(r+1)+L]$, we can recover both positive roots of $f_r(x)$ to the nearest integer using $\lg L = O(\log N)$ evaluations of $f_r(x)$, where each evaluation takes a constant number of basic operations on numbers of size $O(N^{r+1})$. If any integer roots are found, then we check if their gcd's with $N$ are nontrivial, which is quick to do (even naively) in $O((\log N)^2\log \log N)$ bit operations. 
Put together, the total number of bit operations required is $O(N^{1/(r+2)}(\log N)^2\log\log N)$, as claimed.
\end{proof}

\section{The $r = 2$ Case}\label{r=2 case}

The bulk of the work in Algorithm \ref{LehmansMethodGen} when $r=2$ is in determining if $f_2(x) = x^3-ux^2+4ab^2N$ has integer roots. For a given prime $s$, there are $s^2$ polynomials of the form $x^3+cx^2+d \in \zz/s\zz[x]$. For each tuple $(s,c,d)$, we store $1$ in a lookup table if $x^3+cx^2+d$ has a root mod $s$, and we store $0$ otherwise. Considering the number of irreducible cubics in $\zz/s\zz$ is $s^3/3+O(s)$, to help offset the cost of root-finding, we check $f_2(x)$ for integer roots mod $s$ for the first $M_N$ primes, where $M_N$ is chosen so that $\left(2/3\right)^{M_N} < 1/\log N$. Thus, $M_N$ exhibits $\log \log N$ growth with $N$.

We typically chose $M_N =35$ for the $N$ that we tested, so $
\left(2/3\right)^{35} \leq 10^{-6}$.
Heuristically, this means we expect to attempt searching for integer roots of $x^3-ux^2+4ab^2N$ via a bisection method, when none exist, roughly one in a million times. We found that, indeed, a search in a lookup table first (to check reducibility) is much faster than just applying the bisection root-finding uniformly throughout. 
We also note that root finding methods with high rates of convergence like Newton's method and the Secant method were not practical for finding the integer roots of $f_2(x)$. This was due to the close proximity of the two positive roots of $f_2(x)$. 

We remark that Algorithm~\ref{LehmansMethodGen} is structured as a nested {\tt for} loop, and most of the operations occur in the inner-most loop. It is therefore informative to count the total number of {\tt for} loops under the worst-case scenario. When $r=2$, we find this is $\le 660 N^{1/4}$. Therefore, an estimate for when we would expect the generalized Lehman method with $r=2$ to overtake the wheel variant is when $N$ is about
\[
\left(660\cdot \frac{15}{4} \right)^{12} \approx 5.28 \cdot 10^{40}.
\]
This is close to our computed time in Table~\ref{timings table}, where we believe the additional $4$ powers of $10$ come from log factors ignored in the above computation.

Finally, a few details about the modification of Algorithm~\ref{LehmansMethodGen} to solve Problem $\textrm{P}^*$ here, and quickly explain how PS and wheel were used to solve $\textrm{P}^*$. 

Recall from equation $\eqref{Ubound}$ that for a given $a$ and $b$, that the number of $u$'s that must be considered is 
\[
\le \frac{2^{2+3/2}}{\sqrt{ab}}\frac{q^{3/2}}{2B^2p^{1/2}}.
\]
Under the assumption that $q < p < 8q$ (or equivalently $N^{1/3} < p < 2N^{1/3}$) this is 
\[
\le \frac{2^{5/2}}{\sqrt{ab}}\frac{N^{1/3}}{B^2}.
\]
Since we only need to consider $(a,b)$ satisfying $b < a < 8b$, the total number of $u$'s that checked is therefore 
\begin{equation}
    O(N^{1/3}B^{-2}\sqrt{8B\cdot B} + B^2).
\end{equation}
Balancing these terms gives that the optimal choice of $B$ should be $B \approx N^{1/9}$, giving an $N^{2/9+o(1)}$ algorithm. 

To compare the three methods to solve problem $P^*$ with $N$ around $10^j$ for $j = 21,\dots,36$ we selected random primes $q$ in the range
 \[
 \frac{1}{4} 10^{j/3} \leq q \leq 10^{j/3}
 \]
 and then randomly selected $p$ in the range
 $q < p < 8q$
 and computed $N = p^2q$. Then we recorded how long it took each method to recover either $p$ or $q$. For each $j$ we performed $100$ trials. The averages of these times are given in Table~\ref{P* table}. 
 
Both wheel factorization and Pollard-Strassen are adapted only slightly to solve $P^*$, by modifying the search for the divisor $n$ of $N$ in the interval $[(1/4)N^{1/3},N^{1/3}]$ as opposed to the full interval $[2,N^{1/3}]$. For wheel factorization, this cuts the runtime by a factor of $3/4$ over proving square-freeness, while for Pollard--Strassen, only a square-root of that factor is saved, namely $\sqrt{3}/2$.

\section{A Variant of the Harvey-Hittmier integer factorization method}\label{HH variant}
Given $N=pq$, the new Harvey and Hittmier factorization method enables factoring $N$ in about $N^{1/5+o(1)}$ time and space. They first find an $\alpha \bmod{N}$ of large multiplicative order, then use $aq+bp \equiv aN+b \bmod p-1$ to write
\begin{equation}\label{CollisionCong}
\alpha^{aq+bp-\lfloor(4abN)^{1/2}\rfloor}\equiv \alpha^{aN+b-\lfloor(4abN)^{1/2}\rfloor}\mod p.
\end{equation}
This converts the problem of factoring $N$ to solving the discrete log problem \eqref{CollisionCong}. Even though $p$ is unknown, we can find solutions to  \eqref{CollisionCong} by using a novel idea involving fast polynomial multi-evaluation. In turn, the discrete log problem can be solved using the babystep-giantstep algorithm. The idea is to choose $a$ and $b$ suitably, via lattice techniques, to ensure that $aq+bp -(4abN)^{1/2}$ is relatively small. In other words, the left side of $(\ref{CollisionCong})$ is $\alpha^i$ for $i$ small and will be among the computed babysteps. 

The number of babysteps that must be considered is given by 
\[
\lambda := \left \lceil \frac{4N^{1/2}}{m_0^{3/2}}\right\rceil
\]
for a parameter $m_0$. The search for $p$ is then organized by performing their main search that will find a divisor of $N$ in the interval $\left(\left(1+\frac{1}{m_0}\right)^{n-1} ,\left(1+\frac{1}{m_0}\right)^{n}\right]$ if one exists. 

The $1/5$ exponent in their algorithm arises on balancing the role of $m_0$. When $m_0$ is large, the number of babysteps $\lambda$ is small, but we will need to search more intervals in order to find a nontrivial divisor. In total, the main search for prime factors of $N$ runs in time  
\[
O\left( m_0 \lg^4 N + \lambda \lg^2 N\right).
\]
Choosing $m_0 \approx \lambda$ gives that both are on the order of $N^{1/5}$. 

Our adaptation is similar in spirit to that of Lehman's method. The main idea being that for $N = p^rq$, for appropriately chosen $a,b$, the arithmetic mean $aq+rbp$ will be close to $(r+1)(ab^rN)^{1/(r+1)}$. Specifically, we will show below that for $r$th power detection, we can take 
\[
\lambda \approx \frac{N^{5/(6+4r)}}{m_0^{3/2}}.
\]
balancing this with the original $m_0$ term will give an $N^{1/(3+2r)+o(1)}$ algorithm in both time and space for $r$th power detection.

\subsection{Lehman's Inequality for $\lambda$}
The following proposition follows Proposition $3.3$ of \cite{LogLogSpeedUp} with $N/\sigma$ replaced with $N/\sigma^{r-1}$, and a larger multiplicative constant. Since $\sigma_0$ can be taken to be at least $N^{1/(r+2)}$, this accounts for the algorithm's power saving. 
\begin{Prop}\label{LatticeGenPow}
	There exists an algorithm with the following properties. It takes as input positive integers $N >1, m_0,\sigma_0$, a positive integer $m$ coprime to $N,r$, and an integer $\sigma$ coprime to $m$ with $1\leq \sigma \leq m$. It's output is a pair of integers $(a,b)\neq (0,0)$ such that if $N = p^rq$, 
	\[
	\sigma_0 \leq p < \left(1+\frac{1}{m_0}\right)\sigma_0
	\]
	and 
	\[
	p\equiv \sigma \mod m
	\]
	then 
	\begin{equation}\label{LehmanLatticeInequalityGen}
	\left| aq+rbp -\left(a\frac{N}{\sigma_0^r} + rb\sigma_0\right)\right|\leq 2\left(r\binom{r+2}{2}\right)^{1/2}\frac{(N/\sigma_0^{r-1})^{1/2}}{m_0^{3/2}}
	\end{equation}
	and \[
	aq + rbp \equiv a\frac{N}{\sigma^r} + rb\sigma \mod m^2.
	\]
	Assuming that $m_0,\sigma_0,m = O(N)$, the algorithm runs in $O(\lg^2 N)$ bit operations, and the outputs satisfy $|a|,|b| = O(N^2)$.
\end{Prop}
\begin{sproof}
 Follow the proof of Proposition $3.3$ of \cite{LogLogSpeedUp}, but take 
 \[
 q = \frac{N}{p^r} = \frac{N}{\sigma_0^r}(1+t_0)^{-r} = \frac{N}{\sigma_0^r}\left(1-rt_0+\binom{r+2}{2}\delta t_0^2\right)
 \]
 for some $\delta \in [0,1)$. Then obtain the bounds 
\begin{equation}\label{LatticeParallelogramGen}
	\left|-ra\frac{N}{\sigma_0^r} + rb\sigma_0\right| \leq \frac{\left(r\binom{r+2}{2} \right)^{1/2}N^{1/2}m^{1/2}}{m_0^{1/2}\sigma_0^{(r-1)/2}} \quad \quad \text{and} \quad \quad \left|\binom{r+2}{2}a\frac{N}{\sigma_0^r}\right| \leq \frac{ \left( r \binom{r+2}{2}\right)^{1/2}N^{1/2}m^{1/2}m_0^{1/2}}{\sigma_0^{(r-1)/2}}
	\end{equation}
	using the LLL algorithm. 
\end{sproof}

\subsection{Finding Elements of Large Order}
A necessary component of the Harvey-Hittmier method is an element $\alpha$ of large multiplicative order mod $N$. This is to guarantee that any collision between babysteps and giantsteps will give a nontrivial gcd with $N$. The main idea is to sequentially check elements $\alpha = 2,3,\dots$ for the sufficient order, and bound the number of times we can fail. Specifically, each time we fail, we recover information about $p-1$ and $q-1$. This process continues until we find an element of large order or recover enough information to factor $N$. In practice, this is not difficult since almost every element will have large order mod $N$, the difficulty is in doing so deterministically. The following is only a minor alteration of Theorem $6.3$ in \cite{HittBSGS}.

\begin{Lemma}\label{LargeOrder}
	There is an algorithm taking as an input $N$ and $D$ such that $N^{2/(3+2r)}\leq D \leq N$, and returns either some $\alpha \in \zz_N^*$ such that $\operatorname{ord}_{N}(\alpha) > D$, or a nontrivial factor of $N$, or determines if $N$ is $r$ power free. Its running time is 
	\[
	O(D^{1/2}\lg^2 N)
	\]
\end{Lemma}

\subsection{The Main Algorithm}
We present now the main search. Compare to Algorithm $4.3$ of \cite{LogLogSpeedUp}.
\begin{Alg}\label{MainSearchGen} (Main Search)
	\textit{Input:}
	\begin{itemize}
		\item[-] A positive integer $N \geq 2$ with at most $r+1$ prime factors (counting multiplicity) that is not an $r$th power, each factor greater than $N^{2/(3+2r)}$.
		\item[-] Positive integers $m_0$ and $m$ such that $(m,N) = 1$, and $(m,r) = 1$.
		\item[-] An element $\beta\in\zz_N^*$ such that $\operatorname{ord}_N(\beta^{m^2})\geq 2\lambda +1$, where 
		\[
		\lambda := \left \lceil 2\left(r \binom{r+2}{2}\right)^{1/2}\frac{ N^{5/2(3+2r)}}{(mm_0)^{3/2}}\right\rceil
		\]
	\end{itemize}
	\textit{Output:} If $N$ is $r$th power free returns ``$N$ is $r$th power free'' otherwise, a nontrivial factor of $N$.
	\begin{algorithmic}[1]
		\FOR{$i=0,\dots,2\lambda$}
		\STATE compute $\beta^{m^2i} \mod N$.
		\IF{$(N,\beta^{m^2i}-1) \not \in \{1,N\}$} 
		\STATE recover a nontrivial factor\ENDIF \ENDFOR
		\FOR{$\sigma = 1,\dots,m$}
		\IF{$(\sigma,m) = 1$}
		\STATE Initialise $\sigma_0 := \lfloor N^{2/(3+2r)}\rfloor$.
		\WHILE{$\sigma_0 < N^{(1+2r)/r(3+2r)}$}
		\STATE Apply Proposition \ref{LatticeGenPow} with input $N,m_0,\sigma_0,m,\sigma$ to obtain a pair $(a,b) = (a_{\sigma,\sigma_0},b_{\sigma,\sigma_0}).$
		\STATE Compute 
		\[
		j_{\sigma,\sigma_0} := m^2(\tau_0-\lambda)+\tau,
		\]
		where 
		\[
		\tau_0 := \left \lfloor \left(a\frac{N}{\sigma_0^2}+2b\sigma_0\right)/m^2\right \rfloor,
		\]
		and where $\tau$ is the unique integer in $0\leq \tau < m^2$ such that 
		\[
		\tau \equiv a\frac{N}{\sigma^2}+2b\sigma\mod m^2.
		\]
		\STATE Compute 
		\[
		v_{\sigma,\sigma_0} := \beta^{aN+r b-j_{\sigma,\sigma_0}}.
		\]
		\STATE Update $\sigma_{0} := \lceil (1+1/m_0) \sigma_0\rceil$.
		\ENDWHILE
		\ENDIF
		\ENDFOR
		\STATE Applying a sort-and-match algorithm to the babysteps and giantsteps computed in Steps $2$ and $11$, find $i \in \{0,\dots,2\lambda\}$ and all pairs $(\sigma,\sigma_0)$ such that 
		\[
		\beta^{m^2i} \equiv v_{\sigma,\sigma_0} \mod N.
		\]
		For each such match, test the candidates 
		\[
		u:= m^2i+j_{\sigma,\sigma_0},
		\]
		together with $a:= a_{\sigma,\sigma_0}$ and $b:= b_{\sigma,\sigma_0}$ by determining if $f_r(x)$ has an integer root. If $p$ and $q$ are found, return.
		\STATE Let $v_1,\dots,v_n$ be the list of giantsteps $v_{\sigma,\sigma_0}$ computed in Step $11$, skipping those that were discovered in step $13$ to be equal to one of the babysteps. Apply Algorithm $4.1$ of \cite{LogLogSpeedUp} (finding collisions) with $N,\kappa:=2\lambda+1,\alpha:=\beta^{m^2} \mod N$ and $v_1,\dots,v_n$ as inputs. If successful, recover a nontrivial factor.
		\STATE Return ``$N$ is $r$th power free''.
	\end{algorithmic}
\end{Alg}
\begin{Prop}
	Algorithm \ref{MainSearchGen} is correct. If $m = O(N)$, and $m_0 = N^{1/(3+2r)}$ then it runs in 
	\[
	O(M(N) \left(m\lg N (\lg \lg N)^2 + \phi(m)N^{1/(3+2r)}\lg^4 N\right))
	\]
	bit operations.
\end{Prop}
    \begin{proof}
    The proof follows Proposition $4.4$ in \cite{LogLogSpeedUp}. The time save comes from the smaller size of $\lambda$, and balancing $m_0$ accordingly.
\end{proof}

\section{Concluding remarks}

If $N=p^2q$ where $q>1$ is squarefree, then one can use a special squarefree detection algorithm like \cite{booker-hiary-keating} to obtain a lower bound on $q$. Such information about $q$, and hence about $p$, can help improve the actual run time of our algorithm. 

Additionally, there are several consequential improvements that could be made to our fairly basic implementation of Algorithm~\ref{LehmansMethodGen} for $r=2$. These include using a more efficient method to iterate through fractions $a/b$ with $a$ and $b$ relatively prime. In the current implementation, we compute $\operatorname{gcd}(a,b)$ for each pair $(a,b)$, and if the $\operatorname{gcd}$ is nontrivial then we may skip line $7$ in Algorithm~\ref{LehmansMethodGen}. Thus, there are roughly $40N^{1/4}$ gcd computations that could be converted into simple additions, multiplications or integer divisions. 

Finally, the bound \eqref{Pbound} on the length of the interval containing candidate $u$'s could be improved significantly. We already do this in our implementation for $r=2$ by replacing the bound \eqref{Pbound} with the tighter inequality
\begin{equation}\label{Pbound2}
    \frac{P(bp,aq)}{aq+bp} = \frac{aq+8bp}{aq+2bp} < 4. 
\end{equation}
But given the extra information we have about $p$ at each step of Algorithm~\ref{LehmansMethodGen}, namely that $2^{j}N^{1/4}< p < 2^{j+1}N^{1/4}$, the simple inequality \eqref{Pbound2} could be tightened yet some more.

\bibliographystyle{plain}
\bibliography{A_Variation_on_Lehmans_Method_for_Detecting_Square-Free_Integers}
\end{document}